\documentclass{article}
\usepackage[left=3cm,right=3cm,top=3cm,bottom=3cm]{geometry}
\usepackage{amsmath}
\usepackage{mathtools}
\usepackage{amsfonts}
\usepackage{amssymb}
\usepackage{amsthm}

\newcommand{\fd}{\mathrm{fd}}
\newcommand{\rnk}{\mathrm{rk}}
\usepackage[english]{babel}
\usepackage{blkarray, bigstrut}
\usepackage{enumerate}
\usepackage[arrow, matrix, curve]{xy}
\usepackage{amssymb}
\usepackage{amsmath}

\newtheorem{thm}{Theorem}[section]
\newtheorem{lem}[thm]{Lemma}

\newtheorem{varcon}[thm]{Conjecture}

\newtheorem{cor}[thm]{Corollary}
\newtheorem{prop}[thm]{Proposition}

\newtheorem{rem}[thm]{Remark}

\newcommand{\ra}{\rightarrow}

\newcommand{\la}{\leftarrow}

\theoremstyle{definition}
\newtheorem{defn}[thm]{Definition}
\newtheorem{ex}[thm]{Example}

\providecommand{\dim}{\mathop{\rm dim}\nolimits}
\providecommand{\rnk}{\mathop{\rm rk}\nolimits}
\providecommand{\id}{\mathop{\rm id}\nolimits}
\providecommand{\End}{\mathop{\rm End}\nolimits}
\providecommand{\im}{\mathop{\rm im}\nolimits}
\providecommand{\ker}{\mathop{\rm ker}\nolimits}
\providecommand{\coker}{\mathop{\rm coker}\nolimits}

\begin{document}

\title{New bounds on the toral rank with application to cohomologically symplectic spaces}
\author{Leopold Zoller}
\maketitle

\begin{abstract}
We use Boij-S\"oderberg theory to give two lower bounds for the dimension of the cohomology of a finite CW-complex in terms of the toral rank and certain Betti numbers of the space. One of our bounds turns out to be particularly effective for c-symplectic spaces, proving the toral rank conjecture for c-symplectic spaces of formal dimension $\leq 8$.
\end{abstract}

\section{Introduction}

Understanding the topological consequences of symmetries on a space is a classical problem in algebraic topology. One of the central conjectures in this area is the toral rank conjecture (TRC),  which has been a driving force in the field for over three decades. It was first posed by Steve Halperin (see \cite{RHATA}) and revolves around the toral rank $\rnk(X)$ of a space $X$, which is the highest integer $r$ such that the $r$-dimensional torus $T^r=(S^1)^r$ acts almost freely (in a continuous fashion) on $X$. By almost freely we mean that all isotropy groups are finite. The TRC states:
\begin{varcon}
Let $X$ be a finite CW-complex. Then
\[\dim_\mathbb{Q} H^*(X;\mathbb{Q})\geq 2^{\rnk(X)}.\]
\end{varcon}
Or in other words, in order for a torus to act almost freely on a finite CW-complex, it is necessary that the dimension of the total cohomology of the space is at least as large as that of the torus. \\
Often $\rnk(X)$ is replaced by $\rnk_0(X)$, which is the maximum of $\rnk(Y)$ among the finite CW-complexes $Y$ in the rational homotopy type of $X$. While both versions of the conjecture are equivalent, $\rnk_0$ has the advantage of being computable by purely algebraic means in the language of rational homotopy theory and commutative differential graded algebras (cdga):

\begin{prop}[{\cite[Prop. 4.2]{LOMM}}]\label{modelprop}
Let $X$ be a path connected, nilpotent finite CW-complex and $(\Lambda V,d)$ its Sullivan minimal model. Then $\rnk_0(X)$ is the maximum integer $r$ such that there exists an extension sequence
\[(\Lambda (x_1,\ldots,x_r),0)\ra (\Lambda (x_1,\ldots,x_r)\otimes \Lambda V,D)\ra (\Lambda V,d)\]
of cdgas such that the middle cdga has finite dimensional cohomology, the $x_i$ are of degree $2$, and the maps are given by inclusion and projection.
\end{prop}

Here, $\Lambda V$ (resp.\ $\Lambda(x_i)$) denotes the free commutative graded algebra on a graded vector space $V$ (resp.\ the graded vector space generated by the $x_i$). We obtain a purely algebraic reformulation of the TRC, making it accessible to the tools of homological and commutative algebra.\\
While the TRC is still open in general, it has not been untouched: It was proved for special types of spaces like e.g.\ Hard-Lefschetz manifolds in \cite{BOTTR} and elliptic spaces with pure Sullivan models in \cite{RHATA}, just to name a few. Those results were later generalized to c-symplectic spaces of Lefschetz-type in \cite{CSSTAATGG}, \cite{TGSTAATM} and to ellitptic spaces with a certain type of two-stage minimal model in \cite{FTAATSS}. On the other hand, there has also been success in bounding the dimension of the cohomology in the general case: One of the classical results is given by the inequality
\[\dim H^*(X;\mathbb{Q})\geq\begin{cases} 2\rnk_0(X) &\text{for }\rnk_0(X)\leq 2\\
 2(\rnk_0(X)+1) &\text{for }\rnk_0(X)> 2\end{cases},\]
from \cite{BOTTR}, which holds for any finite CW-complex. This has been improved to
\[\dim H^*(X;\mathbb{Q})\geq 2(\rnk_0(X)+\lfloor \rnk_0(X)/3\rfloor),\]
in \cite{CCOAFTA} which is the best general lower bound known to the author.\\
We shall be concerned with both types of approaches, the specific and the general one. Specifically, we will discuss torus actions on spaces that have the cohomology of a compact symplectic manifold.

\begin{defn}
A space $M$ is called c-symplectic if its cohomology algebra is finite dimensional with maximal degree $2n$, satisfies Poincar\'e duality, and there is an element $\omega\in H^2(M;\mathbb{Q})$ such that $\omega^n$ is nontrivial. It is further said to be of Lefschetz-type if multiplication with $\omega^{n-1}$ induces an isomorphism from $H^1(M;\mathbb{Q})$ to $H^{2n-1}(M;\mathbb{Q})$.
\end{defn}

The most important examples of c-symplectic spaces are of course compact symplectic manifolds. However, the notion of a c-symplectic space is a little more general. For instance certain connected sums of copies of $\mathbb{C}\mathrm{P}^n$ are known to not carry almost complex structures and thus, in particular, can not be symplectic. As our arguments will not make use of any geometric structure, we will work with c-symplectic spaces, having compact symplectic manifolds in mind as the main application. In \cite{CSSTAATGG}, concepts from symplectic geometry are imitated topologically, which leads to quite a deep understanding of actions on c-symplectic spaces. Another reference to mention in this context is \cite{TGSTAATM} (based on unpublished notes from 1978) operating on the elegant level of equivariant cohomology. Both references prove

\begin{thm}\label{A}
Let $M$ be a finite CW-complex that is c-symplectic of Lefschetz-type and has an almost free $T^r$-action. Then $H^*(M;\mathbb{Q})\cong H^*(T^r;\mathbb{Q})\otimes H^*(M/T^r;\mathbb{Q})$ as algebras. In particular, the toral rank conjecture holds for $M$.
\end{thm}

Without the Lefschetz-type assumption, a weaker version of the above result still holds in the form of the theorem below from \cite{CSSTAATGG}. While not stated explicitly, the core idea (namely the case of a circle-action) is also given in \cite{TGSTAATM}.

\begin{thm}\label{B}
Let $M$ be a c-symplectic finite CW-complex. Then
\[\rnk_0(M)\leq b_1,\]
where $b_1$ is the first Betti number of $M$.
\end{thm}

Though we can not prove the TRC for c-symplectic spaces that are not of Lefschetz-type, Theorem \ref{B} motivates us to still give an improved lower bound for the dimension of their cohomologies. This is achieved by proving a general lower bound for arbitrary spaces that takes into account the first Betti number and the formal dimension, where the formal dimension $\fd(X)$ of a space $X$ with finite-dimensional cohomology is defined to be the highest degree $n$ such that $H^n(X;\mathbb{Q})\neq 0$. The underlying idea is to investigate what happens when not only the degree $0$ cohomology, but multiple generators in the minimal Hirsch-Brown model of an action get mapped to $0$. With the help of the theory surrounding the rather recently solved conjectures of Boij and S\"oderberg on free resolutions of graded modules, we prove:

\begin{thm}\label{C}
Let $X$ be a finite CW-complex of formal dimension $n$ with an almost free $T^r$-action.
\begin{enumerate}[(i)]
\item Let $b$ be the first Betti number of $X$. Then
\[\dim H^*(X)\geq \min_{k=0,\ldots,b} \frac{n+r-1}{n-r+1}2k+2^{b-k}.\]
\item Let $k$ be the degree corresponding to the first nontrivial odd Betti number. Then
\[\dim H^*(X)\geq \frac{n+r-1}{n-r+1}2\dim H^{<k}(X).\]
\end{enumerate}

\end{thm}

While this is, of course, not a strict improvement of the general lower bound due to the dependence on certain Betti numbers, it still ameliorates the known bounds in many cases. In particular, together with Theorem \ref{B} we get an improved lower bound for c-symplectic spaces, that approaches a quadratic bound for smaller cohomogenities without needing additional assumptions on Betti numbers. This will prove the TRC for c-symplectic spaces of formal dimension $\leq 8$.\\
The material is organised as follows: We begin with a short recap of basic concepts and facts from equivariant cohomology that we will need later. As all necessary tools are at hand, we show how Theorem \ref{B} can be deduced in the style of \cite{TGSTAATM}. The next section develops the necessary algebraic tools for the proof of Theorem \ref{C}. In particular, a very brief introduction to Boij-S\"oderberg theory is given. In section 4, we prove Theorem \ref{C} and give some numbers concerning the effectiveness of our bounds. The final section 5 is dedicated to the implications of Theorems \ref{B} and \ref{C} for c-symplectic spaces. We put some effort into pushing the numbers beyond the immediate corollary to obtain the TRC in dimension $\leq 8$.\\
In what follows, all coefficients will be in $\mathbb{Q}$. When speaking of the $r$-torus $T^r$, we will always assume $r\geq 1$. All spaces are assumed to be path connected.\\
The author wants to thank Christopher Allday for his comments on an earlier version of this paper.

\section{Equivariant cohomology and c-symplectic spaces}

Our essential tool for extracting topological information from a group action is the Borel fibration. For any Lie-group $G$ acting on a space $X$, we can consider the universal $G$-bundle $G\ra EG\ra BG$ and the associated principal bundle \[X\ra (EG\times X)/G\ra BG.\]
The latter is called the Borel fibration of the action and its total space ($EG\times X)/G$ is denoted by $X_G$. The homotopy type of $X_G$ encodes information on the $G$-action on $X$. The following theorem of Hsiang is of fundamental importance to us:

\begin{thm}[\cite{CTOTTG}]
Let $G$ be a compact Lie-group acting on a finite CW-complex $X$. Then the action is almost free if and only if $H^*(X_G)$ is finite dimensional.
\end{thm}

This result can be refined a little more with respect to the notion of formal dimension as defined in the introduction.

\begin{lem}\label{fdlem}
Let $X$ be a finite CW-complex with an almost free action of a compact Lie-group $G$. Then $\fd(X_G)=\fd(X)-\fd(G)$.
\end{lem}

\begin{proof}
We see this by taking a step back in the Puppe-sequence: Consider the pullback of the pathspace fibration over $BG$ along the Borel fibration $p\colon X_G\ra BG$. Its total space is the homotopy fiber $F_p$ of $p$ and thus homotopy equivalent to $X$. Furthermore, we have a fibration
\[\Omega BG\ra F_p\ra X_G\]
where the fundamental group of $X_G$ acts trivially on the homology of the fibers because the fibration is the pullback of a fibration over a simply connected base. In the associated Serre spectral sequence, the second page equals $H^*(\Omega BG)\otimes H^*(X_G)\cong H^*(G)\otimes H^*(X_G)$. As both factors are finite dimensional by assumption, the entry that is the tensor product of both top degree cohomologies lives to infinity. Hence, $H^{*}(X)$ is non-trivial in degree $\fd(X_G)+\fd(G)$.
\end{proof}

\begin{lem}\label{torlem}
Let $V\subset H^2(BT^r)$ be a $k$-dimensional subspace. Then there exists a $(r-k)$-dimensional subtorus $T^{r-k}\subset T^r$ such that the kernel of the induced  map $H^2(BT^r)\ra H^2(BT^{r-k})$ is precisely $V$.
\end{lem}

\begin{proof}
Consider the universal $T^r$-bundle $T^r\ra ET^r\ra BT^r$. As $ET^r$ is contractible, the transgression $d_2\colon H^1(T^r)\ra H^2(BT^r)$ on the second page of the associated Serre spectral sequence is an isomorphism. There is a $(r-k)$-dimensional subtorus $T^{r-k}\subset T^r$ such that the kernel of the map $H^1(T^r)\ra H^1(T^{r-k})$ induced by the inclusion is exactly $d_2^{-1}(V)$. We have a morphism of the two corresponding universal bundles

\[\xymatrix{T^{r-k}\ar[d]\ar[r]&ET^r\ar[d]\ar[r]&ET^r/T^{r-k}\ar[d]\\
T^r\ar[r]&ET^r\ar[r]&ET^r/T^r}\]
which induces a morphism of the associated spectral sequences. Since both transgressions on the second page are isomorphisms, we see that the map induced by $BT^{r-k}=ET^r/T^{r-k}\ra ET^r/T^r=BT^r$ on cohomology has the desired properties.
\end{proof}

Before we come to the proof of Theorem \ref{B}, we want to point out that the proof works by reduction to the case of a circle-action. The latter has been treated before in a refined way in \cite[Prop. 1.6.3]{TGSTAATM} using the same tools as we do below.

\begin{proof}[Proof of Theorem \ref{B}]
Assume there is a $T^r$-action on $M$ where $r>b_1$. Let $a_1,\ldots,a_{b_1}$ be a basis of $H^1(M)$, denote by $d_2$ the differential on the second page of the Serre spectral sequence of the Borel fibration associated to the action and by $\omega$ the symplectic class in $H^2(M)$. We have
\[d_2(\omega)=\sum_{i=1}^{b_1}a_i\otimes p_i\]
for certain $p_1,\ldots,p_{b_1}\in H^2(BT^r)$. By Lemma \ref{torlem}, $r>b_1$ implies the existence of a sub-circle $S^1\subset T^r$ such that the $p_i$ lie in the kernel of the map $H^*(BT^r)\ra H^*(BS^1)$. The morphism
\[\xymatrix{M\ar[r]\ar[d]^{\id} & M_{T^r}\ar[r]\ar[d]& BT^r\ar[d]\\
M\ar[r]&M_{S^1}\ar[r]&BS^1}\]
of the Borel fibrations of the $T^r$- and the restricted $S^1$-actions induces a morphism of the associated spectral sequences. Therefore, by construction, the second differential of the spectral sequence associated to the lower row vanishes on $\omega$. Observe that the same holds for all subsequent differentials due to degree reasons. This implies that $\omega^n$ lives to infinity, where $2n$ is the formal dimension of $M$. But this contradicts the fact that the formal dimension of $M_{S^1}$ is less than $2n$.
\end{proof}

\begin{rem}
To put the above proof into a geometric context, note that we essentially show that, in case of an almost free action, the element $d_2(\omega)\in H^2(BT^r)\otimes H^1(M)\cong \mathfrak{t}^*\otimes H^1(M)\cong \hom(\mathfrak{t},H^1(M))$ has to be injective, where $\mathfrak{t}$ denotes the Lie-algebra of $T^r$. Compare this to the situation of a symplectic action on a smooth symplectic manifold $(M,\omega)$: Any element in the kernel of the homomorphism $\mathfrak{t}\ra H^1(M)$ sending $X$ to the contraction of $\omega$ along the fundamental vector field of $X$ generates a subgroup of $T^r$ that acts in a hamiltonian fashion and thus has fixed points, preventing the action to be almost free. The connection between the two homomorphisms can be made explicit using the Cartan-model.
\end{rem}

\begin{cor}\label{kleineabsch}
Let $M$ be a c-symplectic finite CW-complex with $\fd(M)=2n\geq 4$. Then
\[\dim H^*(M)\geq 4\rnk_0(M).\]
\end{cor}

\begin{proof}
By Theorem \ref{B}, we have $\dim H^1(M)\geq \rnk_0(M)$. As $H^*(M)$ fulfils Poincar\'e duality, the same is true for $\dim H^{2n-1}(M)$. If $\rnk_0(M)\geq 1$, then the Euler characteristic of $M$ is equal to $0$. Both $1$ and $2n-1$ are odd so the corollary follows.
\end{proof}

With regards to the toral rank conjecture one might hope that the statement of Theorem \ref{B} can be improved in the spirit of Theorem \ref{A} where we find linearly independet elements in $H^1(M)$ that span an exterior algebra. In fact it can be shown that, under the Lefschetz-type assumption, the inclusion of an orbit induces a surjection $H^*(M)\ra H^*(T^r)$. To see that the latter fails in the more general case one needs to look no further than the standard $T^2$-action on the Kodaira-Thurston manifold. The example below makes it clear that we can not hope to directly obtain an improved lower bound on $\dim H^*(M)$ by using the algebra structure of $H^*(M)$.

\begin{ex} Consider the cdga $(\Lambda(a_1,a_2,a_3,b_1,b_2,b_3),d)$ where all generators are of degree $1$ and $d(a_1)=d(a_2)=d(a_3)=0$, $d(b_1)=a_2a_3$, $d(b_2)=a_3a_1$ and $d(b_3)=a_1a_2$. This is the Sullivan model of a compact nilmanifold $M$. The element $\omega=a_1b_2+a_2b_3+a_3b_1$ fulfils $d(\omega)=0$ and $\omega^3=-6a_1a_2a_3b_1b_2b_3$ so $M$ is c-symplectic. Notice that $H^1(M)$ is spanned by the classes of the $a_i$ and that the cohomology class of any element of the form $a_ia_j$ is trivial in $H^2(M)$. Thus every two elements of $H^1(M)$ have trivial product. Nonetheless, we have $rk_0(M)\geq 3$:
\\Consider the extension $A=(\Lambda(X_1,X_2,X_3,a_1,a_2,a_3,b_1,b_2,b_3),D)$ where $D(X_i)=D(a_i)=0$ and $D(b_i)=d(b_i)+X_i$. The cohomology classes of the $X_i$ are identified with the classes of the nilpotent elements $-d(b_i)$. A spectral sequence argument shows that $H^*(A)$ is a finitely generated $\Lambda(X_1,X_2,X_3)$ module so this implies that $H^*(A)$ is finite dimensional. By Proposition \ref{modelprop}, this proves the claim.
\end{ex}

\section{Graded modules}

Let $R$ denote the polynomial ring in $r$ variables of degree 2. If $(\Lambda V,d)$ is the Sullivan minimal model of a space $X$ with a free $T^r$-action, a model for the Borel fibration is given by
\[(R,0)\ra (R\otimes \Lambda V,D)\ra (\Lambda V,d),\]

where the maps are given by the inclusion of $R$ and the projection onto $1\otimes \Lambda V$. In particular, a model for $X_{T^r}$ is given by $R\otimes \Lambda V$ with a twisted differential $D$. We see from the sequence that $D$ vanishes on $R\otimes 1$. Therefore, it is $R$-linear. As we are ultimately interested in $H^*(X)$, we want to replace $\Lambda V$ by its cohomology in the above twisted tensor product. To do this, one has to give up on the cdga structure and enter the realm of differential graded $R$-modules. Before we dive into their theory, let us quickly discuss the above construction, which is commonly referred to as the minimal Hirsch-Brown model of the action.

\subsection{Constructing the minimal Hirsch-Brown model via perturbations}\label{HBM}
The minimal Hirsch-Brown Model is constructed in great detail in Appendix B of \cite{CMITG}, using homotopy theory of graded modules. The reason we discuss the construction here is that we want to have close control over the passage from the Sullivan to the Hirsch-Brown model. While this can also be achieved using the standard construction, we will choose a slightly different approach and make use of Gugenheims theory of perturbations from \cite{OTCCOAF}. Consider the following situation:\\
We have two differential $R$-modules $(H,d_H)$ and $(M,d_M)$, where $H$ is a retract of $M$ in the sense that we find $f\colon H\ra M$ and $g\colon M\ra H$ with $g\circ f=\id_H$ and $f\circ g\simeq\id_M$ via a homotopy $\phi$. We further assume $\phi$ fulfils the side conditions $\phi^2=0$, $\phi f=0$,  and $g\phi=0$. We are interested in the following problem: Given a new differential $D_M$ on $M$, find a new differential on $H$ such that the two new differential modules are again homotopy equivalent. For all $n\geq 1$ we define
\[t:=D_M-d_M,\qquad t_n:=(t\phi)^{n-1}t,\qquad \Sigma_n:=t_1+\ldots+t_n,\]
and
\begin{align*}
\delta_{n+1}&:=d_H+g\Sigma_n f
&f_{n+1}:=f+\phi\Sigma_n f\\
g_{n+1}&:=g+g\Sigma_n\phi
&\phi_{n+1}:=\phi+\phi\Sigma_n\phi.
\end{align*}
Let $\mathcal{A}\subset\End(M)$ be the (non-commutative) algebra consisting of maps that arise as polynomials  in the operators $\phi$, $t$, and $d_M$. Let $J\subset \mathcal{A}$ denote the ideal generated by $t$.

\begin{lem}[\cite{OTCCOAF}]\label{perturblem} We have
\begin{align*}
&\delta_n\delta_n\in g J^n f, &D_Mf_n-f_n\delta_n\in J^n f,\\ & \delta_ng-gD_M\in g J^n,
&g_nf_n=\id_H,\\
&f_ng_n-\id_M-D_M\phi_n-\phi_n D_M\in J^n, &\phi_nf_n=0,\\ &g_n\phi_n=0, &\phi_n\phi_n=0.
\end{align*}
\end{lem}

So if, in a pointwise sense, the above sequences of maps have a limit and $J^n$ converges to $0$, the limits solve the problem described above. Note that while no gradings are used in the above definitions, the constructions behave well with respect to gradings: If we start with the data of a retract of graded differential modules with differentials of degree $1$, where maps and homotopies are graded (of degree $0$ and $-1$), the newly obtained homotopy equivalence will also respect the grading. Now let us apply this to the case described in the beginning of this section.\\
Choose the data of a retract (maps and the homotopy) of graded differential $\mathbb{Q}$-modules between $(H^*(X),0)$ and $(\Lambda V,d)$ such that all conditions of the scenario described above are met (we will explicitly do this in the proof of Theorem \ref{C}). Now extend all the maps $R$-linearly to a homotopy equivalence of the differential graded $R$-modules $(R\otimes H^*(X),0)$ and $(R\otimes\Lambda V,1\otimes d)$. As we are not interested in $1\otimes d$ but in the twisted differential $D$ on $R\otimes\Lambda V$, set $t=D-1\otimes d$. Observe that since $D$ and $1\otimes d$ agree on the component mapping $1\otimes \Lambda V$ to itself, $t$ maps $R\otimes(\Lambda V)^{n}$ into $R\otimes (\Lambda V)^{\leq n-1}$. All operators in $\mathcal{A}$ preserve $R\otimes(\Lambda V)^{\leq n}$, so this implies that the operators in $J^{n+1}$ vanish on $R\otimes (\Lambda V)^{\leq n}$. In particular, on this domain, we have $t_{n+1}=0$ and $\Sigma_n=\Sigma_{n+1}$. Thus all sequences of maps in Lemma \ref{perturblem} converge (pointwise) to respective limits, which by the lemma define a differential $\delta$ on $R\otimes H^*(X)$ and a homotopy equivalence between $(R\otimes H^*(X),\delta)$ and $(R\otimes \Lambda V,D)$. On $R\otimes H^{\leq n}(X)$, the differential is given by
\[\delta=g\Sigma_nf.\]
Let $I\subset R$ denote the maximal homogeneous ideal. Note that since $t$ has values in $I\otimes \Lambda V$, $\delta$ takes values in $I\otimes H^*(X)$ as well.

\subsection{Boij-S\"oderberg theory}

While the main results from Boij-S\"oderberg theory have more refined and more general versions than what is presented here, we will focus on what we want to apply to the minimal Hirsch-Brown model later in the proof of Theorem \ref{C}. For a more complete picture, see e.g.\ \cite{BSTIAS}. Let $R$ be a polynomial ring in $r$ variables over some field. To stick to the usual conventions, we consider the variables of $R$ to be of degree $1$. \\
Let us briefly recall the basic facts about graded free resolutions (see e.g.\ the first chapter of \cite{TGOS}). For any finitely generated graded $R$-module $M=\bigoplus_k M^k$ and integer $n$, we denote by $M(n)$ the graded module with $M(n)^k=M^{n+k}$. A free resolution of $M$ is an exact complex
\[0\la M\la F_0\la F_1\la\ldots\]
of free modules $F_i=\bigoplus R(-j)^{\beta_{i,j}}$ and degree $0$ maps. $M$ has a unique minimal free resolution that is a direct summand of any free resolution of $M$. It is characterized by the fact that at each stage, the image of the map $F_i\la F_{i+1}$ is contained in $IF_i$, where $I$ is the maximal homogeneous ideal in $R$. The integers $\beta_{i,j}$ in the minimal free resolution are called the graded Betti numbers of $M$. The length of the minimal free resolution of $M$ is at most $r$ (meaning that $F_i=0$ for $i>r$) and only finitely many of the Betti numbers are non-zero. Hence, we can see the collection of the $\beta_{i,j}$ as an element of $\bigoplus_{j\in\mathbb{Z}}\mathbb{Z}^{r+1}\subset\mathbb{D}:=\bigoplus_{j\in\mathbb{Z}}\mathbb{Q}^{r+1}$. This element is called the Betti diagram of $M$.

\begin{ex}
For $R=\mathbb{Q}[x,y]$, consider the module $M=R/(x,y^2)$. A free resolution is given by
\[0\la M\la R\xleftarrow{\begin{pmatrix}
x&y^2
\end{pmatrix}} R(-1)\oplus R(-2)\xleftarrow{\begin{pmatrix}
y^2\\-x
\end{pmatrix}} R(-3)\la 0\]
As the maps between the free modules in the resolution have image in the multiples of $I$, this is the minimal free resolution of $M$. If we display the corresponding Betti diagram $(\beta_{i,j})\in\mathbb{D}$ as an array, showing only the window of $\mathbb{D}$ where $(\beta_{i,j})$ is non-trivial, we obtain
\setlength{\bigstrutjot}{3pt}
\[
\begin{blockarray}{cccc}
\begin{block}{cccc}
 &0&1&2\\
\end{block}
\begin{block}{c [ccc]}
0 & 1 & 0 & 0\bigstrut[t]\\
1 & 0 & 1 & 0\\
2 & 0 & 1 & 0\\
3 & 0 & 0 & 1\bigstrut[b]\\
\end{block}
\end{blockarray},\]
where $\beta_{i,j}$ is located in the $i$th column and $j$th row.
\end{ex}

A very important class of modules is given by those whose depth is equal to their dimension, which by the Auslander-Buchsbaum formula is equivalent to the fact that their codimension coincides with the length of their minimal free resolution. Those modules are called Cohen-Macaulay.\\
The main result of Boij-S\"oderberg theory is to classify all possible Betti diagrams of Cohen-Macaulay modules up to multiplication by a rational number. Let us make this result precise.

\begin{defn}
\begin{enumerate}[(i)]
\item A degree sequence of length $c$ is an element $d=(d_0,\ldots,d_c)\in\mathbb{Z}^{c+1}$ such that $d_0<\ldots<d_c$.
\item Let $d$ be a degree sequence of length $c$. We say that a finitely generated graded $R$-module has a pure diagram of type $d$ if its Betti numbers satisfy $\beta_{i,j}=0$ whenever $i>c$ or
$j\neq d_i$.
\end{enumerate}
\end{defn}

A digram is pure if and only if it has only one non-trivial entry in each column when displayed as above. Given a degree sequence $d$ of length $c$, we define the associated pure diagram $\pi(d)\in\mathbb{D}$ by
\[\pi(d)_{i,j}=\begin{cases} \prod_{k\neq i} \frac{1}{\vert d_k-d_i\vert}& \text{if $0\leq i\leq c$ and $j=d_i$}\\
0 &\text{else}\end{cases}.\]

One can prove that any Betti diagram of a codimension $c$ Cohen-Macaulay Module with pure resolution of type $d$ is a rational multiple of $\pi(d)$. This is a consequence of certain restrictions on Betti diagrams known as the \emph{Herzog-K\"uhl equations}, which simplify to the above formulas in the pure case. Thus pure diagrams are understood up to multiplication by a rational number. The Boij-S\"oderberg conjectures enable us to extend this understanding to the Betti diagrams of arbitrary Cohen-Macaulay modules. They were proved in \cite{BNOGMACOV} and state:

\begin{thm}\label{Boij-Soder}\begin{enumerate}[(i)]
\item The Betti diagram of a Cohen-Macaulay Module of codimension $c$ is a positive rational linear combination of Betti diagrams of codimension $c$ Cohen-Macaulay Modules with pure diagrams.
\item For any degree sequence $d$, there is a Cohen-Macaulay module with pure diagram of type $d$.
\end{enumerate}
\end{thm}

Note that the characteristic zero case of $(ii)$ was proved earlier in \cite{TEOPFR}. Together, the two statements yield a precise description of the rational cone spanned by the Betti diagrams of all Cohen-Macaulay modules in $\mathbb{D}$. However, the problem which integral points in this cone actually are the Betti diagram of a module is unsolved.

\section{Proof of Theorem \ref{C}}
Again, denote by $R$ the rational polynomial ring in $r\geq 1$ variables and by $I$ its maximal homogeneous ideal.
One of the classical results regarding the toral rank is the inequality
\[\dim H^*(X)\geq 2rk_0(X)\]
by Allday and Puppe. The idea of the proof is that since $R\otimes H^0(X)$ gets mapped to $0$ in the minimal Hirsch-Brown model of a $T^r$-action, almost all of $R\otimes H^0(X)$ has to get killed in cohomology in order for the cohomology of $X_{T^r}$ to be finite. By projecting to this submodule, the differential induces a map
\[R\otimes H^{odd}\ra R\otimes H^0(X)\] with image in $I\otimes H_0(X)$, whose cokernel is finite dimensional over $\mathbb{Q}$. One can show, e.g.\ by means of the Krull height theorem, that such a map $R^l\ra R$ requires $l$ to be at least $r$. From this, the above inequality can be deduced with help of the Euler characteristic.\\
Our goal is to investigate how the above result can be improved if there are multiple copies of $R$ getting mapped to $0$ in the minimal Hirsch-Brown model. In other words: If we have a map $R^l\ra R^k$ with image in $IR^k$ and finite dimensional cokernel, what can be said about the relation of $k$ and $l$? One can show that $l\geq k+r-1$ is necessary and e.g.\ the matrices
\[\begin{pmatrix}
0&X&Y\\X&Y&0
\end{pmatrix},\quad \begin{pmatrix}
0&W&X&Y&Z\\
W&X&Y&Z&0
\end{pmatrix},\quad \begin{pmatrix}
0&0&X&Y&Z\\0&X&Y&Z&0\\X&Y&Z&0&0
\end{pmatrix}\]
suggest that this bound is actually sharp, as they have finite dimensional cokernel. While this is not as effective as one might have hoped, the above result can be improved if we bring the formal dimension of the cokernel into the picture. For the moment, we make the convention that $R$ is generated in degree $0$ with variables of degree $1$.

\begin{prop}\label{abschatzprop}
Let $k,l\geq 1$ and $f\colon R^l\ra R^k$ a graded $R$-linear map with respect to some grading on $R^l$ and the grading with generators concentrated in degree $0$ on $R^k$. Assume further that
 $\im(f)\subset IR^k$ and that $\coker(f)$ has finite length. Let $N$ be the maximal degree in which $\coker(f)$ is non-trivial. Then
\[l\geq \frac{N+r}{N+1}k.\]
\end{prop}

\begin{proof}
We give $R^l$ the grading that turns $f$ into a degree $0$ map of graded $R$-modules. Consider the minimal free resolution $\im(f)\la F_0\la F_1\la\ldots$ of $\im(f)$. As minimality is equivalent to the image of the map $F_{i}\la F_{i+1}$ being contained in $IF_i$ at each stage, we deduce that
\[\coker(f)\la R^k\la F_0\la F_1\la\ldots\]
is a minimal free resolution of $\coker(f)$. For the graded Betti numbers $\beta_{i,j}$ of $\coker(f)$, we set $\beta_i=\sum_j\beta_{i,j}$. By minimality of the resolution, we deduce that $R^l$ contains $F_0$ as a direct summand so $l\geq \beta_1$ and $k=\beta_0$.\\
We have a closer look at the Betti diagram of $\coker(f)$. The Castelnuovo-Mumford regularity of $\coker(f)$ is equal to $N$ (c.f.\ \cite[Cor.\ 4.4]{TGOS}). This means that $\beta_{i,j}=0$ for $j>N+i$. Furthermore, note that the Krull dimension and the depth of $\coker(f)$ are equal to $0$ because $\coker(f)$ has finite length. So $\coker(f)$ is a Cohen-Macaulay module of codimension $r$. By Theorem \ref{Boij-Soder}, the Betti diagram of $\coker(f)$ can be written as a positive linear combination of pure diagrams of length $r$. This implies that a lower bound on the ratio of the first two Betti numbers of such a pure diagram gives a lower bound for $\beta_1/\beta_0$ and thus for $l/k$.\\
Let $b_0,\ldots,b_r$ be the Betti numbers of a codimension $r$ Cohen-Macaulay module $C$ with pure Betti diagram of type $d$, where $d$ is a degree sequence of length $r$. We can assume that $d_0=0$ and $d_i\leq N+i$, for orherwise the Betti diagram of $C$ can not occur nontrivially in a positive linear combination forming the Betti diagram of $\coker(f)$ due to the degree restrictions above. The Herzog-K\"uhl equations yield
\[\frac{b_1}{b_0}=\prod_{i=2}^r \frac{d_i}{d_i-d_1}.\]
This ratio is minimal when the $d_i$ take their maximum values $d_{i}=N+i$ and $d_1$ takes its minimum value $1$. In total, we get

\[\frac{b_1}{b_0}\geq\prod_{i=2}^{r}\frac{N+i}{N+i-1}=\frac{N+r}{N+1},\]
which proves the claim.
\end{proof}

Note that this bound is sharp: By part $(ii)$ of Theorem \ref{Boij-Soder}, for any $N\geq 0$, there is a codimension $r$ Cohen-Macaulay module with pure diagram of type $d$, where $d_0=0$, $d_1=1$ and $d_i=N+i$ for $2\leq i\leq r$. This module is of finite length and its maximal non-trivial degree coincides with the Castelnuovo-Mumford regularity $N$. Consequently, the first map in the free resolution has the desired properties.\\
The question that ensues is when we actually have generators in the minimal Hirsch-Brown model that map to 0. The following trick, used in the proof of the TRC for Hard-Lefschetz manifolds in \cite{BOTTR}, will be fundamental to assume the existence of such generators.

\begin{lem}\label{extalglem}
Let $X$ be a finite CW-complex with an almost free $T^r$-action and 
\[(R,0)\ra(R\otimes \Lambda V,D)\ra (\Lambda V,d)\]
a Sullivan model for the associated Borel fibration. If $Z\subset V^1$ is a subspace such that $d(Z)=0$ and $D|_{1\otimes Z}$ is injective, then $Z$ generates an exterior algebra in $H^*(X)$. In particular, if $\fd(X)-1\leq r\leq \fd(X)$, then $H^*(X)$ contains an exterior algebra on $r$-generators of degree $1$.
\end{lem}

\begin{proof}
We show the first part via induction.  Assume that $\im(d)\cap \Lambda^l Z=0$ and $d(x)\in \Lambda^{l+1} Z$ for some $x\in\Lambda V$. We can write
$D(x)=d(x)+y$ for some $y\in R^{\geq 2}\otimes \Lambda V$. By assumption, $D(Z)$ is contained in $R^2\otimes 1$ so $D$ maps $d(x)$ into $R^2\otimes\Lambda^lZ$. As $D^2(x)=0$, it follows that $D(y)=-D(d(x))\in R^2\otimes\Lambda^lZ$. But as the $R^2\otimes \Lambda V$ component of $D(y)$ is, for degree reasons, equal to $(1\otimes d)(y)$ this implies $D(y)=0$ by the induction hypothesis. Hence, $D(d(x))=0$ and also $d(x)=0$ because $D$ is injective on $\Lambda Z$. By induction, we obtain that $\im(d)\cap \Lambda Z$ is trivial so $\Lambda Z$ projects injectively into cohomology.\\
Now, if $\fd(X)-1\leq r\leq \fd(X)$, then $\fd(X_{T^r})\leq 1$ so $D$ must map surjectively onto $R^2\otimes 1$, which is $r$-dimensional. Preimages must, for degree reasons, lie in $V^1\cap \ker(d)$ so the second part follows.
\end{proof}

\begin{proof}[Proof of Theorem \ref{C}]
For the proof of $(i)$, consider a Sullivan model
\[(R,0)\ra(R\otimes \Lambda V,D)\ra (\Lambda V,d)\] for the Borel fibration of the action. By assumption, $\ker(d|_{V^1})$ is $b$-dimensional. Now decompose $\ker(d|_{V^1})=Z\oplus Z'$, where $Z'=\ker(D|_{V^1})$ and let $k$ be the dimension of $Z'$. By Lemma \ref{extalglem}, $Z$ generates an exterior algebra in cohomology. If $k=0$, Theorem \ref{C} holds so in what follows we will assume $k\geq 1$. If $r=n$, then $H^1(X_{T^r})=0$, which implies $k=0$ so we will assume $r<n$ as well.\\
Let us now construct the minimal Hirsch-Brown model of the action. We can decompose $\Lambda V=A\oplus B\oplus C$ as vector spaces, where $d$ vanishes on $A$ and $B$ and maps $C$ isomorphically onto $B$. Note that since $\Lambda Z\oplus Z'$ projects injectively into cohomology, the decomposition above can be chosen in a way that $\Lambda Z\oplus Z'$ is contained in $A$. Choosing cycles from $A$ (respectively projecting onto cohomology) gives an isomorphism $H^*(X)\cong A$. Sticking to the notation in section \ref{HBM} we define
\begin{align*}
f&\colon H^*(X)\cong A\ra \Lambda V\\
g&\colon \Lambda V\ra A\cong H^*(X)\\
\phi&\colon \Lambda V\ra B\xrightarrow{d^{-1}} C\ra \Lambda V,
\end{align*}
where all non-specified arrows correspond to the inclusions and projections with respect to the decomposition. One checks that this is the data of a retract between $H^*(X)$ and $\Lambda V$, fulfilling the requirements made in the construction of the Hirsch-Brown model. Let $\delta$ denote the corresponding differential on $R\otimes H^*(X)$. On $R\otimes H^{\leq l}(X)$, it is explicitly given by
\[\delta=g\Sigma_l f.\]
Note that since $t$ maps $\Lambda Z\oplus Z'$ to itself and thus into $A$, we have $\phi t|_{\Lambda Z\oplus Z'}=0$ and in particular $\delta|_{\Lambda Z\oplus Z'}=gtf=gDf$.
This implies $\delta(R\otimes \overline{\Lambda Z})\subset R\otimes\overline{\Lambda Z}$, where $\overline{\Lambda Z}\subset H^*(X)$ is the exterior algebra induced by $\Lambda Z$, and $\delta(R\otimes\overline{Z'})=0$, $\overline{Z'}$ being defined analogously.\\
Hence, when composed with a suitable projection onto $R\otimes \overline{Z'}$, $\delta$ induces a map
\[R\otimes (H^*(X)/\overline{\Lambda Z})^{even}\ra R\otimes \overline{Z'}\]
with image in $I\otimes \overline{Z'}$, whose cokernel is finite dimensional because the cohomology of the Hirsch-Brown model is. The formal dimension of $X_{T^r}$ is $n-r$. This implies that when we give $R\otimes \overline{Z'}$ the grading from Proposition \ref{abschatzprop}, where generators are in degree $0$ and the variables of $R$ have degree $1$, the maximal degree $N$ in which the cokernel is non-trivial fulfils $2N+1\leq n-r$, so $N\leq \frac{n-r-1}{2}$. Proposition \ref{abschatzprop} yields
\[\dim (H^*(X)/\overline{\Lambda Z})^{even}\geq \frac{n+r-1}{n-r+1}k.\]
As $\overline{\Lambda Z}$ lives equally in odd and even degrees, we obtain
\[\dim H^*(X)^{even}\geq \frac{n+r-1}{n-r+1}k +2^{b-k-1}.\]
Note that the case $b=k$ is special but still true. Using that the Euler characteristic of $X$ is $0$ completes the proof of $(i)$.\\
In the situation of $(ii)$, note that in a minimal Hirsch-Brown model of an almost free $T^r$-action, $R\otimes H^{<k}(X)$ gets mapped to 0 for degree reasons. Projecting onto this submodule, the differential induces a map
\[R\otimes H^{odd}(X)\ra R\otimes H^{<k}(X)\]
which has image in $I\otimes H^{<k}(X)$ and finite dimensional cokernel as the cohomology of $X_T$ is finite dimensional. Applying Proposition \ref{abschatzprop} as in $(i)$ together with the fact that the Euler characteristic is $0$ yields the desired lower bound.
\end{proof}

Let us take some time  to put Theorem \ref{C} into perspective and compare it to the known results. The quotient $(n+r-1)/(n-r+1)$ approaches $2r-1$ as $r$ approaches $n$. In this sense, for a fixed space $X$, we can say that our lower bound in $(ii)$ approaches a linear bound of slope $4\dim H^{<k}(X)$ for smaller cohomogenities.\\
For the bound in $(i)$, calculations show that for any $a,b\in\mathbb{R}$ the real valued function $f(x)=ax+2^{b-x}$ has global minimum $a(b-\log_2(a)+(1+\log\log(2))/\log(2))$. Thus, in a vague sense, approaching the extreme case $r=n$, the bound approximates to a linear bound with slope close to $4(b-\log_2(n))$.
However, note that in both cases, the approximation occurs rather late and that the result is not interesting for very small cohomogenities, such as $n-1\leq r\leq n$, where the TRC is more or less trivial as seen in Lemma \ref{extalglem} above. Still, the theorem gives an improved lower bound in many cases.\\
The tables below are meant to give a feeling for the behaviour of our lower bounds and when Theorem \ref{C} is an improvement of the established linear bound of slope $8/3$ by Amann mentioned in the introduction. In both tables, we have set $n=10$. Let us begin with estimate $(i)$:\\
\begin{center}
\begin{tabular}{c|ccccccccccc}
$r$&1&2&3&4&5&6&7&8&9&10\\
\hline $b=4$ &8&9&10&12&13&14&16&16&16&16\\
\hline $b=6$&12&14&16&19&22&26&32&39&50&64\\
\hline $b=10$&20&24&28&34&41&50&64&84&122&216
\end{tabular}
\end{center}~\\\\
Now let us have have a look at scenario $(ii)$ of Theorem \ref{C} where we set $l=\dim H^{<k}(X)$.\\
\begin{center}
\begin{tabular}{c|cccccccccc}
$r$&1&2&3&4&5&6&7&8&9&10\\

\hline $l=4$
&8&10&12&15&19&24&32&46&72&152\\
\hline $l=6$&12&15&18&23&28&36&48&68&108&228\\
\hline $l=10$
&20&25&30&38&47&60&80&114&180&380
\end{tabular}\\

\end{center}

\section{Lower bounds for c-symplectic spaces}

The motivating examples for Theorem \ref{C} are, of course, c-symplectic spaces, where it immediately combines with Theorem \ref{B} to give 
\[\dim H^*(M)\geq \min_{k=0,\ldots,r} \frac{2n+r-1}{2n-r+1}2k+2^{r-k}\]
for a finite, c-symplectic CW-complex $M$ of formal dimension $2n$ with almost free $T^r$-action (one quickly checks that formula $(i)$ from Theorem \ref{C} grows monotonously in $b$ so by Theorem $B$, we can replace $b$ by $r$). This bound, in contrast to Theorem \ref{C}, requires no assumptions on Betti numbers and, in the same spirit as the discussion of the previous bounds above, approaches a quadratic bound for small cohomogenities. We can push this result a little more if we add Poincar\'e duality to the picture.

\begin{thm}\label{symplabsch}
Let $M$ be a c-symplectic finite CW-complex of formal dimension $2n$ with an almost free $T^r$-action. Then
\begin{enumerate}[(i)]
\item \[\dim H^*(M)\geq\min_{k=0,\ldots,r}\frac{2n+r-1}{2n-r+1}4k+4\sum_{i=0}^{\frac{n-1}{2}}\binom{r-k}{2i}\]
if $n$ is odd and $r\geq n+1$.
\item \[\dim H^*(M)\geq\min_{k=0,\ldots,r}\frac{2n+r-1}{2n-r+1}4k+4\sum_{i=0}^{\frac{n-2}{2}}\binom{r-k}{2i}+2\binom{r-k}{n}\]
if $n$ is even and $r\geq n$.
\end{enumerate}
In particular, the TRC holds for $M$ if $n\leq 4$.
\end{thm}

\begin{rem}
Calculations show that, for $n\geq 7$, the minimum of the bounds in Theorem \ref{symplabsch} is realized for $k>r-n$. In particular, in both the even and the odd bound, we can replace the sum of the binomial coefficients by $2^{r-k+1}$. So both bounds coincide and are exactly double the bound mentioned at the beginning of this section. However, for $n\leq 6$ the distinction is necessary.
\end{rem}

\begin{proof}
The proof is, for the most part, identical to the proof of Theorem $C$ but we pay more attention to the degrees of elements to be able to use Poincar\'e duality. Let $k,Z,Z'$ be as in the proof of Theorem $C$ and construct the map
\[p:R\otimes (H^*(M)/\overline{\Lambda Z})^{even}\ra R\otimes \overline{Z'}\]
as before. Again, the case $k=0$ (and thus also $r=2n$) must be treated separately: $H^*(M)$ contains an exterior algebra on $r$ generators of degree $1$ and fulfils Poincar\'e duality  so the theorem holds. Thus in what follows we will assume $k>0$ and $r<2n$.\\
Now instead of just applying Proposition \ref{abschatzprop}, let us refine the argument a little. As in the proof of Proposition \ref{abschatzprop}, we can consider $p$ as the first step in a free resolution of $\coker(p)$. Now decompose $R\otimes (H^*(M)/\overline{\Lambda Z})^{even}=F\oplus F'$, where $F,F'$ are free $R$ modules and $p|_F:F\ra R\otimes \overline{Z'}$ is the first stage of a minimal free resolution of $\coker(p)$. The space of generators $F\otimes \mathbb{Q}\cong F/IF$ of $F$ injects into $(H^*(M)/\overline{\Lambda Z})^{even}$. An element of $F$ that is of degree $l$ with respect to the free resolution grading as in Proposition \ref{abschatzprop} (meaning that variables of $R$ have degree $1$, $\overline{Z'}$ is concentrated in degree $0$ and $p$ has degree $0$) maps to an element of degree $l$ in the free resolution grading of $R\otimes \overline{Z'}$. In the cdga grading, $\overline{Z'}$ has degree $1$ and variables are of degree $2$ so this corresponds to degree $2l+1$. Taking into account that $p$ is induced by the Hirsch-Brown differential, which is of degree 1, this implies that the degree $l$ component of $F\otimes\mathbb{Q}$ (with free resolution grading) injects into the degree $2l$ component of $(H^*(M)/\overline{\Lambda Z})^{even}$ (with cdga grading).\\
Now assume first that $n$ is odd. The Castelnuovo-Mumford regularity $N$ of $\coker(p)$ satisfies $N\leq (2n-r-1)/2\leq(n-2)/2$. In particular, the generators of $F$ lie in degrees $\leq n/2$ and thus contribute to $(H^*(M)/\overline{\Lambda Z})^{<n,~even}$. Analogous to the proof of Theorem \ref{C}, adding the part of $\overline{\Lambda Z}$ that lies in even degrees below $n$, we obtain
\[\dim H^{<n,~even}(M)\geq \frac{2n+r-1}{2n-r+1}k+\sum_{i=0}^{\frac{n-1}{2}}\binom{r-k}{2i}.\]
Now the lower bound for $\dim H^*(M)$ follows by using Poincar\'e duality and the Euler characteristic to multiply this by $4$.\\
In case $n$ is even, we have to pay extra attention to degree $n$ because it can not be doubled using Poincar\'e duality. First, we deal with the case when $n\leq r\leq n+1$. The case $r\geq n+2$ will be exposed as an easy special case. For some $d_1<\frac{2n-r+3}{2}$, define
\[S(d_1):=\prod_{i=2}^{r} \frac{2n-r-1+2i}{2n-r-1+2i-2 d_1}.\] Recall that for a codimension $r$ Cohen-Macaulay module of Castelnuovo-Mumford regularity $\leq\frac{2n-r-1}{2}$ with pure Betti diagram of type $d$, we can bound the ratio of the first two Betti numbers by
\[\frac{\beta_1}{\beta_0}\geq S(d_1)\] as in the proof of Proposition \ref{abschatzprop}. The Betti diagram of $\coker(p)$ decomposes as a positive linear combination $a_1\pi(d^1)+\ldots+a_l\pi(d^l)$ of pure diagrams. The sum of all $a_i\pi(d^i)_{0,0}$ equals $k$, the zeroth Betti number of $\coker(p)$. Define $\alpha$ to be the sum of those $a_i\pi(d^i)_{0,0}$ for which $d^i_1=n/2$. Note that $d_i>n/2$ is not possible due to the degree restrictions of the regularity. At the first stage of the minimal free resolution of $\coker(p)$ we obtain
\[\dim(F\otimes \mathbb{Q})^\frac{n}{2}\geq S\left(\frac{n}{2}\right)\alpha\quad\text{and}\quad\dim(F\otimes \mathbb{Q})^{<\frac{n}{2}}\geq S(1)(k-\alpha)\]
because $S(1)\leq S(d_i)$ for any of the $d_i$. As above, adding $\overline{\Lambda Z}$, we obtain
\[\dim H^n(M)\geq S\left(\frac{n}{2}\right)\alpha+\binom{r-k}{n}\]and\[\dim H^{<n,~even}(M)\geq S(1)(k-\alpha)+\sum_{i=0}^{\frac{n-2}{2}}\binom{r-k}{2i}.\]
When bounding all of $H^*(M)$, we can count the $<n$ part twice, due to Poincar\'e duality and then double everything using the Euler characteristic. In total, we obtain
\[\dim H^*(M)\geq  S(1)4(k-\alpha)+S\left(\frac{n}{2}\right)2\alpha+4\sum_{i=0}^{\frac{n-2}{2}}\binom{r-k}{2i}+2\binom{r-k}{n}.\]
Note that, for $\alpha=0$, this is the bound claimed in the theorem. For $n\leq r\leq n+1$ and $n\geq 4$, Lemma \ref{rumrechnen} below shows that $S\left(\frac{n}{2}\right)\geq 2S(1)$ so the above bound is minimal for $\alpha=0$ and  the theorem holds. When $n=2$, the desired bound is dominated by Corollary \ref{kleineabsch}. If $r\geq n+2$, $\alpha=0$ must hold from the start because of the degree restrictions of the regularity. To verify that this (in combination with Corollary \ref{kleineabsch}) proves the TRC for $n\leq 4$, see Remark \ref{bembem} below.
\end{proof}

For the remaining values of $r$, the bound can not be doubled with the help of Poincar\'e duality. However, we still get a rather complicated, intermediate result that lies between the original bound and the bound from Theorem \ref{symplabsch}. Let us do some calculations first, where $S(\cdot)$ is defined as in the proof of Theorem \ref{symplabsch}.

\begin{lem}\label{rumrechnen}
Let $n\geq 4$ even and $3\leq r\leq n+1$. Then $S\left(\frac{n}{2}\right)\geq 2S(1)$.
\end{lem}

\begin{proof}
Observe first that
$S(1)=\frac{2n+r-1}{2n-r+1}\leq \frac{3n}{n}= 3$ and that
\[S\left(\frac{n}{2}\right)=\prod_{i=2}^{r}\frac{2n-r-1+2i}{n-r-1+2i}\geq\prod_{i=2}^{r}\frac{2n+r-1}{n+r-1}\geq\left(\frac{3}{2}\right)^{r-1}.\]
This implies that, independent of $n$, the claim is true for $r\geq 6$. For the remaining cases, we use induction over $n$. Assume that the claim is true for some $n,r$ which is equivalent to $P(n)\geq 2$ where
\[P(n):=\frac{S\left(\frac{n}{2}\right)}{S(1)}=\prod_{i=2}^{r}\frac{2n-r-3+2i}{n-r-1+2i}.\]
Leaving $r$ fixed, we obtain
\[P(n+2)=\prod_{i=2}^{r}\frac{2n-r-3+2(i+2)}{n-r-1+2(i+1)}=P(n)\frac{f}{g}\] 
where $f=(2n+r+1)(2n+r-1)(n-r+3)$ and $g=(2n-r+1)(2n-r-1)(n+r+1)$. So $P(n+2)\geq P(n)$ is equivalent to $f-g=8n^2+16nr-2r^3+2r^2+2r-2\geq 0$. As $n\geq r-1$, this expression is bounded from below by $-2r^3+26r^2-30r+6$, which is positive for $3\leq r\leq 5$. This implies that if the claim is true for $(n,r)$ with $3\leq r\leq 5$, it is also true for $(n+2,r)$. The lemma now follows by checking that it holds for $(n,r)=(4,3),(4,4),$ and $(4,5)$.
\end{proof}

\begin{thm}\label{unteresymplabsch}
Let $M$ be a c-symplectic finite CW-complex of formal dimension $2n$ with an almost free $T^r$-action where $r\leq n$ if $n$ is odd and $r\leq n-1$ if $n$ is even. Then
\[\dim H^*(M)\geq \min_{\substack{k=0,\ldots,r\\ \gamma\in[0,k]}}\max\left(B_1(k,\gamma),B_2(k,\gamma)\right)\]

where \[B_1(k,\gamma)=2S(1)(k-\gamma)+2S\left(\left\lfloor\frac{n}{2}+1\right\rfloor\right)\gamma+2	^{r-k}\]
and \[B_2(k,\gamma)=4S(1)(k-\gamma)+2^{r-k+1}.\]
\end{thm}

\begin{proof}
Define $k,Z,Z',p,F$ as in the proof of Theorem \ref{symplabsch}. Again, the theorem holds for $k=0$ so assume $k>0$. Let us treat the case when $n$ is even. We decompose the Betti diagram of $\coker(p)$ as a positive linear combination $a_1\pi(d^1)+\ldots+a_l\pi(d^l)$ of pure diagrams. Now define $\alpha$ and $\gamma$ as the sum of those $a_i\pi(d^i)_{0,0}$ for which $d^i_1=\frac{n}{2}$ and $d^i_1\geq\frac{n}{2}+1$. As before, we obtain bounds on $\dim F\otimes\mathbb{Q}$ in certain degrees which translate into the bounds
\begin{align*}
\dim H^{<n,~even}(M) &\geq S(1)(k-\alpha-\gamma)+2^{r-k-1}\\
\dim H^n(M) &\geq S\left( \frac{n}{2} \right)\alpha \\
\dim H^{>n,~even}(M) &\geq S\left(\frac{n}{2}+1\right)\gamma,
\end{align*}
where we have used that the even part of $\overline{\Lambda Z}$ is contained in $H^{<n,~even}(M)$. With the help of the Euler characteristic, we obtain
\[\dim H^*(M)\geq C_1(k,\alpha,\gamma):= 2S(1)(k-\alpha-\gamma)+2S\left( \frac{n}{2} \right)\alpha+2S\left(\frac{n}{2}+1\right)\gamma+2^{r-k}.\]
On the other hand, if we forget about cohomology in degree $>n$ and use Poincar\'e duality first, we get
\[\dim H^*(M)\geq C_2(k,\alpha,\gamma):=4S(1)(k-\alpha-\gamma)+2S\left( \frac{n}{2} \right)\alpha+2^{r-k+1}.\]
This implies $\dim H^*(M)\geq \max\left(C_1(k,\alpha,\gamma),C_2(k,\alpha,\gamma)\right)$. If we can prove that this expression takes its minimum for $\alpha=0$, this proves the claim. For $r\leq 2$, the bounds get dominated by Corollary \ref{kleineabsch} so there is nothing to prove. For $n\geq 4$ and $r\geq 3$, this is a consequence of Lemma \ref{rumrechnen}. Hence, the Theorem is proved in case $n$ is even. If $n$ is odd, the proof is completely analogous except we can assume $\alpha=0$ from the start because $F$ has no elements in degree $\frac{n}{2}$.
\end{proof}

Let us have a look at some values of the lower bounds given by Theorems \ref{symplabsch} and \ref{unteresymplabsch}.\\
\begin{center}
\begin{tabular}{c|cccccccccc}
$r$&1&2&3&4&5&6&7&8&9&10\\
\hline $n=2$&3&6&10&16\\
\hline $n=3$&3&6&12&28&44&64\\
\hline $n=4$&3&7&13&25&40&65&110&214\\
\hline $n=5$&3&6&11&20&33&52&80&123&208&428

\end{tabular}
\end{center}

\begin{rem}\label{bembem} Recall that the TRC is fulfilled for $2n-1\leq r\leq 2n$ by Lemma \ref{extalglem}. Combining this with the values of the above table, it follows that the TRC holds for c-symplectic spaces of formal dimension $8$ or less. It is also interesting to note that, for some small $r$ and $n$, our lower bounds are actually stronger than the TRC.\end{rem}

\bibliography{bibbib}
\bibliographystyle{alpha}
\end{document}